\documentclass[11pt,oneside]{article}
\usepackage{amsfonts}
\usepackage{amsmath}
\usepackage{amsthm}
\usepackage[mathscr]{euscript}
\usepackage{mathrsfs}
\usepackage{indentfirst}
\usepackage{xcolor}
\allowdisplaybreaks[4]
\newcommand{\lbl}[1]{\label{#1}}

\newtheorem{theo}{Theorem}[section]

\newtheorem{lem}{Lemma}[section]

\newcommand{\be}{\begin{equation}}
\newcommand{\ee}{\end{equation}}
\newcommand\bes{\begin{eqnarray}} \newcommand\ees{\end{eqnarray}}
\newcommand{\bess}{\begin{eqnarray*}}
\newcommand{\eess}{\end{eqnarray*}}

\newcommand\kk{\left}
\newcommand\rr{\right}
\newcommand\dd{\displaystyle}

\newcommand\yy{\infty}
\newcommand\ol{\overline}
\newcommand\nn{\nabla}

\newcommand\oo{\Omega}

\newcommand\tr{\Delta}

\newcommand\qq{\eqref}

\begin{document}
\setlength{\baselineskip}{16pt} \pagestyle{myheadings}

\begin{center}{\LARGE\bf Global solvability of a predator-prey model}\\[4mm]
 {\LARGE\bf  with predator-taxis and prey-taxis}\footnote{The first author was supported by {\it Educational research projects for young and middle-aged teachers in Fujian (No. JAT200480)} and {\it Startup Foundation for Advanced Talents of Xiamen University of Technology (No. YKJ20019R)}. The second author was
supported by NSFC Grant 11771110}\\[4mm]
 {\Large  Jianping Wang\footnote{{\sl E-mail}: jianping0215@163.com}}\\[1mm]
{School of Applied Mathematics, Xiamen University of Technology, Xiamen, 361024, China}\\[4mm]
{\Large  Mingxin Wang\footnote{Corresponding author. {\sl E-mail}: mxwang@hit.edu.cn}}\\[1mm]
{School of Mathematics, Harbin Institute of Technology, Harbin 150001, China}
\end{center}

\begin{quote}
\noindent{\bf Abstract.} This paper is concerned with a diffusive predator-prey model with predator-taxis and prey-taxis. Based on the Schauder fixed point theorem, we prove the global existence, uniqueness and boundedness of the classical solutions under the conditions that the predator-taxis and prey-taxis effects are weak enough.

\noindent{\bf Keywords:} Predator-prey model; Predator-taxis; Prey-taxis;  Existence and boundedness.

\noindent {\bf AMS subject classifications (2010)}:
35A01, 35K51, 35K57, 92C17.
 \end{quote}

 \section{Introduction and main results}
 \setcounter{equation}{0} {\setlength\arraycolsep{2pt}

Attractive prey-taxis describes the biological phenomenon that the predator move towards higher concentrations of prey. It was first observed in \cite{Kare} that the ladybugs (predators) in area-restricted search tend to move toward areas with high aphids (prey) density to increase the efficiency of predation. Since the pioneer work of \cite{Kare}, the prey-taxis systems have been widely investigated by many authors. The general form of the prey-taxis system with constant taxis coefficient is
 \bes
 \left\{\begin{array}{lll}
  u_t=d_1\Delta u-\chi\nabla\cdot(u\nabla v)-uh(u)+c_1uF(v),\\[1.5mm]
 v_t=d_2\Delta v+g(v)-uF(v),\\[1.5mm]
 \end{array}\right.\label{1.0}
 \ees
where the unknown functions $u(x,t), v(x,t)$ represent the density of the predator and prey, respectively. The term $h(u)$ describes the mortality rate of predators. The function $g(v)$ is the growth function of prey. $F(v)$ is the functional response function accounting for the intake rate of predators as a function of prey density and the parameter $c_1$ is the conversion rate. The term $-\chi\nabla\cdot(u\nabla v)$ represents the prey-taxis effect, where $\chi$ is a positive constant.

The global existence, uniqueness and boundedness of solutions of \eqref{1.0} have been studied by many authors, see, for example, \cite{HeZ, JinW, Tao, WjpW-CAMWA,W-S-W2017,xiang-RWA} and the references therein. Especially, it was discovered in \cite{JinW,W-S-W2017} that the system is global solvable in two space dimension, while smallness assumption for $\chi$ can prevent blow up in high dimensions. For a parabolic-elliptic version of \eqref{1.0},  the global existence of solutions and global stability of a spatial homogeneous equilibrium were established in \cite{WinklerJDE17}.

Repulsive predator-taxis explains the phenomenon that prey move away from the gradient of predator. The example is the presence of bass (predator) restricts crayfish (prey) foraging and increases anti-predator behaviour such as shelter seeking \cite{Hill}. A typical form of predator-taxis system is
\bess
 \left\{\begin{array}{lll}
 u_t=d_1\Delta u-uh(u)+c_1uF(v),\\[1.5mm]
 v_t=d_2\Delta v+\xi\nabla\cdot(v\nabla u)+g(v)-uF(v),\\[1.5mm]
 \end{array}\right.
 \eess
Here $\xi\nabla\cdot(v\nabla u)$ represents the repulsive predator-taxis mechanism, where the constant $\xi>0$. For $F(v)\le kv$ with $k>0$ and sufficiently small $\xi$, the global existence and boundedness of solutions, existence and stability of coexistence steady state solutions as well as Turing instability are given in \cite{W-W-S2018}.

Assume from now on that $h(u)=a_1+b_1u$, $g(v)=a_2v-b_2v^2$ and $F(v)=v$. By combing the above two taxis mechanisms, it arrives at the following system
\bes
 \left\{\begin{array}{lll}
 u_t=d_1\Delta u-\chi\nabla\cdot(u\nabla v)+u(-a_1-b_1u+c_1v), \ \ &x\in\Omega,\ \ t>0,\\[1mm]
 v_t=d_2\Delta v+\xi\nabla\cdot(v\nabla u)+v(a_2-b_2v-u),&x\in\Omega,\ \ t>0,\\[1mm]
 \dd\frac{\partial u}{\partial\nu}=\dd\frac{\partial v}{\partial\nu}=0,\ \ &x\in\partial\Omega, \ \ t>0,\\[1mm]
 u(x,0)=u_0(x),\ v(x,0)=v_0(x), &x\in\Omega,
 \end{array}\right.\label{1.1}
 \ees
where $\Omega$ is a bounded domain in $\mathbb{R}^n$ with smooth boundary $\partial\Omega$, $\nu$ denotes the outward normal vector on $\partial\Omega$, and constants $\chi,\ \xi,\,a_i,\ b_i,\,c_1>0$, $i=1,2$. System \eqref{1.1} is also referred to as a pursuit-evasion model (\cite{TsyganovBHB2003}).

Despite the well development of the prey-taxis system, the surveys of the predator-prey system with predator-taxis and prey-taxis are at an early stage. In \cite{taow-2020arxiv,taow-2021jfa}, the global existence and large time behavior of weak solutions are constructed in a bounded interval in one space dimension. It is shown in \cite{fuest2020siam} that, under some conditions on the initial data, system \eqref{1.1} admits global classical solutions near homogeneous steady states and these solutions converge to the homogeneous steady states. In \cite{fuest2021arxiv}, global weak solutions to a variant of \eqref{1.1} with nonlinear diffusion and saturated taxis sensitivity are constructed. The spatial pattern formation induced by the prey-taxis and predator-taxis is clarified in \cite{wang-w-s2021}. As it has been stated in \cite{taow-2020arxiv,taow-2021jfa,fuest2020siam}, it is more challenging to analysis \eqref{1.1} compared with the single-taxis system, even for the local existence of solutions. This article proves the global existence of the classical solutions provided the taxis mechanisms are weak enough.

Notations of H\"{o}lder spaces from \cite{Deuring-MZ} and \cite[Chapter 3]{Friedman} are very important for our conclusion. We denote $Q_T=\Omega\times(0,T]$ with $T\in(0,\infty)$ and
 \[\|\cdot\|_p=\|\cdot\|_{L^p(\Omega)}, \ \ \|\cdot\|_{2+\alpha,\,\bar\Omega}
 =\|\cdot\|_{C^{2+\alpha}(\bar\Omega)}, \ \ |\cdot|_{i+\alpha,\,\ol Q_T}=|\cdot|_{C^{i+\alpha,\,\frac{i+\alpha}2}(\ol Q_T)}, \ \ i=0,1,2\]
for the simplicity. Especially, $|\cdot|_{0,\,\ol Q_T}=|\cdot|_{C(\ol Q_T)}$. Throughout this paper the initial data $u_0,v_0$ are supposed to satisfy
 \bes
u_0,\, v_0\in C^{2+\alpha}(\bar\Omega)\ {\rm with}\ \alpha\in(0,1),\ \ \,u_0,\, v_0\ge0\ \ {\rm in}\ \bar\Omega,\ \ \ \dd\frac{\partial u_0}{\partial\nu}=\dd\frac{\partial v_0}{\partial\nu}=0\ \ {\rm on}\ \partial\Omega.\label{1.3a}
 \ees
We state the main result as follows:
\begin{theo}\label{t1}\, Let $n\ge1$. Then there exist $k>0$ and $C>0$ such that for
\bess
0<\chi,\xi<k,
\eess
the problem \qq{1.1} has a unique global solution $(u,v)\in[C^{2,1}(\bar\Omega\times[0,\infty))]^2$, and
 \bes
 \|u(\cdot,t)\|_{C^2(\bar\oo)}+\|v(\cdot,t)\|_{C^2(\bar\oo)}\le C\ \ \ for\ all\ t\in(0,\yy).\label{1.2}
\ees
\end{theo}

The idea of proving Theorem \ref{t1} is inspired by \cite{Deuring-MZ}. Based on the Schauder-type estimates, we first derive a priori estimates. And then by use of the Schauder fixed point theorem, we prove the existence and boundedness of solutions of the problem \eqref{1.1}. Finally, we show the uniqueness. We remark that under the assumption that the mortality rates are density dependent (i.e., $b_1,b_2>0$), one can obtain similar conclusions for other form of \eqref{1.1} via following the arguments leading to Theorem \ref{t1}.

\section{Global existence and boundedness}
\setcounter{equation}{0} {\setlength\arraycolsep{2pt}

\subsection{A basic lemma and some notations}
\begin{lem}{\rm (\cite[Theorem 3.1]{Deuring-MZ})}\label{l1.1}\, {\rm(i)}\,
There is a function $\mathcal K:(0,1)\times(0,\infty)^2\rightarrow(0,\infty)$ with the following property:
\vspace{-2mm}\begin{quote}
Let $\alpha\in(0,1)$, $0<\varepsilon\le k$, $T\in[1,\infty)$, $a,b_i,c,f\in  C^{\alpha,\alpha/2}(\overline Q_T)$, $\phi\in C^{2+\alpha}(\bar\Omega)$, with $\max\{|a|_{\alpha,\,\ol Q_T},|b_i|_{0,\ol Q_T},\,|c|_{0,\ol Q_T}\}\le k$ and $a\ge \varepsilon$ on $\ol Q_T$. If $u\in C^{2,1}(\ol Q_T)$ solves
 \bess
 \left\{\begin{array}{lll}
 u_t-a(x,t)\Delta u+\dd\sum_{i=1}^n b_iD_iu-cu=f,&x\in \Omega,\ t\in(0,T],\\[1mm]
 \dd\frac{\partial u}{\partial\nu}=0,\ \ &x\in \partial\Omega,\ t\in[0,T],\\[2mm]
 u(x,0)=\phi(x), &x\in\bar\Omega,
 \end{array}\right.
 \eess
then
\bess
|u|_{\alpha,\,\ol Q_T}\le \mathcal K(\alpha,k,\varepsilon)\left(|f|_{0,\ol Q_T}+\|\phi\|_{2,\,\bar\Omega}+|u|_{0,\ol Q_T}\right).
\eess\end{quote}

{\rm(ii)}\, There is a function $\mathcal L:(0,1)\times(0,\infty)^2\rightarrow(0,\infty)$ with the following property:
\vspace{-2mm}
\begin{quote}
Let $\alpha,k,\varepsilon,T,a,b_i,c,f,\phi,u$ be given as in {\rm(i)}, but with the assumption: $\max\{|a|_{\alpha,\,\ol Q_T},|b_i|_{0,\,\ol Q_T}$, $|c|_{0,\,\ol Q_T}\}\le k$ replaced by the stronger condition: $\max\{|a|_{\alpha,\,\ol Q_T},|b_i|_{\alpha,\ol Q_T},\,|c|_{\alpha,\ol Q_T}\}\le k$. Then
\bess
|u|_{2+\alpha,\,\ol Q_T}\le \mathcal L(\alpha,k,\varepsilon)\left(|f|_{\alpha,\,\ol Q_T}+\|\phi\|_{2+\alpha,\,\bar\Omega}+|u|_{0,\,\ol Q_T}\right).
\eess\end{quote}
\end{lem}

For the later use, we next introduce some notations. For $0<\alpha<1$ and
\bes\left\{\begin{array}{ll}
\rho=\min\left\{d_1,\,d_2,\,b_1,\,b_2\right\},\\[2mm] \sigma=\max\left\{d_1,\,d_2,\,b_1,\,b_2,\,a_1,\,\dd\frac{3a_2}{\rho},\,
\frac{3c_1}{\rho},\,\|u_0\|_{2+\alpha,\,\bar\Omega}^{\frac12},\,
\|v_0\|_{2+\alpha,\,\bar\Omega}\right\},
\end{array}\right.
\label{2.1}\ees
we define
 \bess
 &h_1=h_1(\rho,\sigma)=\sigma(1+\rho+2\sigma), \ \ h_2=h_2(\rho,\sigma)=\sigma(1+\rho),& \\[1mm]
 &h_3=h_3(\rho,\sigma)=\sigma(1+\rho\sigma+\sigma^2), \ \ h_4=h_4(\rho,\sigma)=\sigma(1+\rho\sigma),&
 \eess
and
 \bes\left\{\begin{array}{ll}
P=2\max\{\mathcal K(\alpha,h_1,\rho), \ \mathcal K(\alpha,h_3,\rho)\},\\[2mm]
R=\min\left\{\dd\frac{3\rho}{\sigma+\sigma^3}, \ \dd\frac{\rho}{(2\sigma^2(1+2P)+2)\mathcal L_2}, \
\dd\frac{\rho}{[\sigma^3(\rho+\sigma)(1+2P)+2\sigma]\mathcal L_4}\right\},
 \end{array}\right.\lbl{2.2d}\ees
where $\mathcal L_2=\mathcal L(\alpha,h_2,\rho)$, $\mathcal L_4=\mathcal L(\alpha,h_4,\rho)$, $\mathcal L,\,\mathcal K$ are determined by Lemma \ref{l1.1}.

\subsection{Existence of solutions}
The following lemma provides a priori estimates for $v$ with given solution component $u$ and small $\xi$.

\begin{lem}\label{l3.1}\, Let $T\in[1,\infty),\alpha\in(0,1)$, $\sigma,\rho,P,R$ be given by \eqref{2.1} and \eqref{2.2d}. Assume that
\bes
0< \xi\le {R}/{3}\label{2.4}
\ees
and $u\in C^{2+\alpha,1+\alpha/2}(\overline Q_T)$ satisfying
\bes
0\le u\le \sigma R,\ |u|_{\alpha,\,\ol Q_T}\le \sigma PR,\ |u|_{2+\alpha,\,\ol Q_T}\le \rho. \label{2.5}
\ees
Then the problem
\bes
\left\{\begin{array}{lll}
 \tilde v_t-d_2\Delta \tilde v-\dd\frac{\sigma\xi}{R}\nabla u\cdot\nabla\tilde v-\left(\frac{\sigma\xi}{R}\Delta u+a_2-\frac{\sigma }{R}u-\frac{\sigma b_2}{R}\tilde v\right)\tilde v=0,&x\in\Omega,\ t\in(0,T],\\[1mm]
\dd\frac{\partial\tilde v}{\partial\nu}=0,&x\in\partial\Omega,\ t\in[0,T],\\[2mm]
 \tilde v(x,0)=\tilde v_0(x):=\dd\frac{R}{\sigma}v_0(x), &x\in\bar\Omega.
 \end{array}\right.\label{3.1a}
 \ees
admits a unique solution $\tilde v\in C^{2+\alpha,1+\alpha/2}(\overline Q_T)$. Moreover,
 \bes
 0\le \tilde v\le R,\ \ \  |\tilde v|_{\alpha,\,\ol Q_T}\le PR,\ \ \ |\tilde v|_{2+\alpha,\,\ol Q_T}\le \rho.\label{3.2}
 \ees
\end{lem}

\begin{proof}
From \eqref{2.1}, it is easy to get
\bes
d_2,b_2\in[\rho,\sigma],\ a_2\in\left[0, \, \rho\sigma/3\right].\label{2.7}
\ees
Since $\tilde v_0(x):=\frac{R}{\sigma}v_0(x)$, by \eqref{1.3a} and \eqref{2.1}, we have
\bes
\tilde v_0\in C^{2+\alpha}(\bar\Omega),\ \dd\frac{\partial\tilde v_0}{\partial\nu}\Big|_{\partial\Omega}=0,\ \tilde v_0\ge0,\ \ \|\tilde v_0\|_{2+\alpha,\,\bar\Omega}\le R.\label{2.7a}
\ees

{\it Step 1: Existence, uniqueness and boundedness}. Thanks to $\tilde v(x,0)=\tilde v_0(x)\ge0$ for $x\in\bar\Omega$, it is easy to see that $\underline{v}(x,t)\equiv0$ is a lower solution of \eqref{3.1a}.

Let $\bar v(x,t)\equiv R$ for $(x,t)\in\overline Q_T$. It follows from \eqref{2.4}, \eqref{2.5} and \eqref{2.7} that
 \[\kk(\frac{\sigma\xi}{R}\Delta u+ a_2-\frac{\sigma u}{R}-\frac{\sigma b_2}{R}\bar v\rr)\bar v\le\kk(\frac{\sigma}{3}|\tr u|+\frac{\sigma\rho}{3}-\sigma\rho\rr)R\le0\ \ \ {\rm in}\ \ \ol Q_T.\]
Moreover, $\tilde v_0\le R$ due to \eqref{2.7a}. Hence, $\bar v$ is an upper solution of \eqref{3.1a}. Making use of the upper and lower solutions method, one can easily obtain the existence and uniqueness of classical solution to \eqref{3.1a}. Hence, the problem \eqref{3.1a} has a unique solution $\tilde v\in C^{2+\alpha,1+\alpha/2}(\overline Q_T)$ and
\bes
0\le\tilde v\le R\ \  {\rm on}\ \ol Q_T. \label{2.8}
\ees

{\it Step 2: The regularity}. For the convenience, we let
$$r(x,t):=\frac{\sigma\xi}{R}\Delta u+ a_2-\frac{\sigma }{R}u-\frac{\sigma b_2}{R}\tilde v.$$
By \eqref{2.4}, \eqref{2.5}, \eqref{2.7} and \eqref{2.8}, we find
\bess
\max\kk\{d_2, \ \frac{\sigma\xi}{R}|\nabla u|_{0,\ol Q_T}, \ |r|_{0,\ol Q_T}\rr\}\le \sigma+\rho\sigma+2\sigma^2=h_1(\rho,\sigma)\ \ {\rm and}\ \ d_2\ge \rho.
\eess
This combined with \eqref{2.8} and \eqref{2.7a} enables us to apply Lemma \ref{l1.1}\,(i) to \eqref{3.1a} to get
 \bes
|\tilde v|_{\alpha,\,\ol Q_T}&\le& \mathcal K(\alpha,h_1(\rho,\sigma),\rho)\kk(\|\tilde v_0\|_{C^2(\bar\Omega)}+|\tilde v|_{0,\ol Q_T}\rr)\nonumber\\[1mm]
 &\le&2\mathcal K(\alpha,h_1(\rho,\sigma),\rho)R\nonumber\\[1mm]
 &\le&PR,\label{2.8a}
  \ees
where we used \eqref{2.2d} in the last step. This shows the second estimate of \eqref{3.2}.

Next, we shall prove the last inequality of \eqref{3.2}. Rewriting \eqref{3.1a} as
\bes
\left\{\begin{array}{lll}
 \tilde v_t-d_2\dd\Delta \tilde v-\frac{\sigma\xi}{R}\nabla u\cdot\nabla\tilde v-\kk(\frac{\sigma\xi}{R}\Delta u+a_2\rr)\tilde v=\kk(-\frac{\sigma}{R}u-\frac{\sigma b_2}{R}\tilde v\rr)\tilde v,\;\;&x\in\Omega,\ t\in(0,T],\\[1mm]
\dd\frac{\partial\tilde v}{\partial\nu}=0,&x\in\partial\Omega,\ t\in[0,T],\\[2mm]
 \tilde v(x,0)=\tilde v_0(x), &x\in\bar\Omega.
 \end{array}\right.\qquad\label{2.9}
 \ees
Making use of \eqref{2.4}, \eqref{2.5} and \eqref{2.7}, it follows that
 \bes
d_2\ge \rho\ \ {\rm and}\ \ \max\kk\{d_2, \ \frac{\sigma\xi}{R}|\nabla u|_{\alpha,\ol Q_T}, \ \kk|\frac{\sigma\xi}{R}\Delta u+a_2\rr|_{\alpha,\,\ol Q_T}\rr\}\le \sigma+\rho\sigma=:h_2(\rho,\sigma).
 \label{3.7}\ees
Noting that
 \bes
|fg|_{\alpha,\,\ol Q_T}\le|f|_{0,\ol Q_T}|g|_{0,\ol Q_T}+|f|_{0,\ol Q_T}|g|_{\alpha,\,\ol Q_T}+|f|_{\alpha,\ol Q_T}|g|_{0,\,\ol Q_T}\label{3.7a}
 \ees
holds for all $f,g\in C^{\alpha,\alpha/2}(\overline Q_T)$. In view of \eqref{3.7a}, \eqref{2.8a}, and \eqref{2.5}, we have
 \bes
\kk|\kk(-\frac{\sigma u}{R}-\frac{\sigma b_2\tilde v}{R}\rr)\tilde v\rr|_{\alpha,\,\ol Q_T}\le 2\sigma^2R(1+2P).\label{3.8}
 \ees
Recalling that $\|\tilde v_0\|_{2+\alpha,\,\bar\Omega}\le R$ due to \eqref{2.7a}. Based on \eqref{3.7} and \eqref{3.8}, using Lemma \ref{l1.1}\,(ii) to \eqref{2.9}, we find
 \bess
|\tilde v|_{2+\alpha,\,\ol Q_T}&\le&\mathcal L(\alpha,h_2(\rho,\sigma),\rho)\kk(\kk|\kk(-\frac{\sigma u}{R}-\frac{\sigma b_2\tilde v}{R}\rr)\tilde v\rr|_{\alpha,\,\ol Q_T}+\|\tilde v_0\|_{2+\alpha,\,\bar\Omega}+|\tilde v|_{0,\ol Q_T}\rr)\nonumber\\[1mm]
&\le&\mathcal L(\alpha,h_2(\rho,\sigma),\rho)(2\sigma^2(1+2P)+2)R\nonumber\\[1mm]
&=&(2\sigma^2(1+2P)+2)\mathcal L_2R\nonumber\\[1mm]
&\le&\rho.
\eess
Here we used \eqref{2.2d} in the last deduction. The proof is end.
\end{proof}

The following lemma provides the a priori estimates for $u$ when $\chi$ is small enough and $v$ is given and satisfies \eqref{3.2}.

\begin{lem}\label{l3.2}\, Let $T\in[1,\infty),\alpha\in(0,1)$, and $\rho,\sigma,P,R$ be given by \eqref{2.1} and \eqref{2.2d}. Assume that
\bes
0< \chi\le{\sigma R}/{3} \label{2.14}
\ees
and $v\in C^{2+\alpha,1+\alpha/2}(\overline Q_T)$ with
\bes
0\le v\le R,\ |v|_{\alpha,\,\ol Q_T}\le PR,\  |v|_{2+\alpha,\,\ol Q_T}\le \rho.\label{2.15}
\ees
Then the problem
 \bes
\left\{\begin{array}{lll}
 \tilde u_t-d_1\Delta\dd\tilde u+\frac{\sigma\chi}{R}\nabla v\cdot\nabla\tilde u-\kk(-\frac{\sigma\chi}{R}\Delta v-a_1+\frac{\sigma c_1}{R}v-\frac{\sigma b_1}{R}\tilde u\rr)\tilde u=0,\ &x\in\Omega,\ \ 0<t\le T,\\[2mm]
 \dd\frac{\partial\tilde u}{\partial\nu}=0,&x\in\partial\Omega, \ \ 0\le t\le T,\\[2mm]
 \tilde u(x,0)=\tilde u_0(x):=\dd\frac{R}{\sigma}u_0,&x\in\bar\Omega
 \end{array}\right.\qquad\label{3.9a}
 \ees
has a unique solution $\tilde u\in C^{2+\alpha,1+\alpha/2}(\overline Q_T)$ which satisfies
 \bes
 0\le\tilde u\le\sigma R,\ \ |\tilde u|_{\alpha,\,\ol Q_T}\le \sigma PR,\ \ |\tilde u|_{2+\alpha,\,\ol Q_T}\le \rho.\label{3.10}
 \ees
\end{lem}

\begin{proof}\, The proof is similar to that of Lemma \ref{l3.1}. It follows from \eqref{2.1} that
\bes
d_1,b_1\in[\rho,\sigma],\ a_1\in[0,\sigma], \ c_1\in\left[0,\,{\rho\sigma}/{3}\right].\label{2.18}
\ees
By \eqref{1.3a} and \eqref{2.1}, we have
\bes
\tilde u_0\in C^{2+\alpha}(\bar\Omega),\ \dd\frac{\partial\tilde u_0}{\partial\nu}\Big|_{\partial\Omega}=0,\ \tilde u_0\ge0,\ \ \|\tilde u_0\|_{2+\alpha,\,\bar\Omega}\le \sigma R.\label{2.18a}
\ees
{\it Step 1: Existence, uniqueness and boundedness.} Due to $\tilde u_0\ge0$ on $\bar\Omega$, it is easily seem that $\underline{u}\equiv0$ is the lower solution of \eqref{3.9a}.

Let $\bar u(x,t)\equiv \sigma R$ in $\ol Q_T$. From \eqref{2.14}, \eqref{2.15} and \eqref{2.18}, we find
\bess
\kk(-\frac{\sigma\chi}{R}\Delta v-a_1+\frac{\sigma c_1}{R}v-\frac{\sigma b_1}{R}\bar u\rr)\bar u
\le\kk(\frac{\sigma^2}{3}|\tr v|+\frac{\rho\sigma^2}{3}-\rho\sigma^2\rr)\sigma R\le 0\ \ \ {\rm in}\;\; Q_T.
\eess
Hence, it is easy to verify that $\bar u$ is an upper solution of \eqref{3.9a}. The upper and lower solutions method shows that problem \eqref{3.9a} admits a unique solution $\tilde u\in C^{2+\alpha,1+\alpha/2}(\overline Q_T)$ which satisfies
\bes
0\le\tilde u\le\sigma R.\label{2.19}
\ees
This establishes the first inequality in \eqref{3.10}.

{\it Step 2: The regularity \eqref{3.10}.} For the convenience, let
\[g(x,t)=-\frac{\sigma\chi}{R}\Delta v-a_1+\frac{\sigma c_1}{R}v-\frac{\sigma b_1}{R}\tilde u.\]
Thanks to \eqref{2.14}, \eqref{2.15}, \eqref{2.18} and \eqref{2.19}, we have
  \bes
d_1\ge\rho\ \ {\rm and}\ \ \max\kk\{d_1,\ \frac{\sigma\chi}{R}|\nabla v|_{0,\ol Q_T},\ |g|_{0,\ol Q_T}\rr\}\le \sigma+\sigma^2\rho+\sigma^3=h_3(\rho,\sigma). \label{2.20}
  \ees
We then use Lemma \ref{l1.1}\,(i) to infer that
\bes
|\tilde u|_{\alpha,\,\ol Q_T}&\le&\mathcal{K}(\alpha,h_3(\rho,\sigma),\rho)(\|\tilde u_0\|_{2,\bar\Omega}+|\tilde u|_{0,\ol Q_T})\nonumber\\[1mm]
&\le&2\mathcal{K}(\alpha,h_3(\rho,\sigma),\rho)\sigma R\nonumber\\[1mm]
&\le&\sigma PR,\label{2.22}
\ees
where we have used \eqref{2.20}, \eqref{2.18a}, \eqref{2.19} and \eqref{2.2d}. This proves the second estimate of \eqref{3.10}.

It remains to show the last inequality of \eqref{3.10}. To achieve this, we need to rewrite \eqref{3.9a} as
 \bes
\left\{\begin{array}{lll}
 \tilde u_t-d_1\Delta\dd\tilde u+\frac{\sigma\chi}{R}\nabla v\cdot\nabla\tilde u-\kk(-\frac{\sigma\chi}{R}\Delta v-a_1\rr)\tilde u=\kk(\frac{\sigma c_1}{R}v-\frac{\sigma b_1}{R}\tilde u\rr)\tilde u,\;\;&x\in\Omega,\ \ 0<t\le T,\\[2mm]
 \dd\frac{\partial\tilde u}{\partial\nu}=0,&x\in\partial\Omega, \ \ 0\le t\le T,\\[2mm]
 \tilde u(x,0)=\tilde u_0(x),&x\in\bar\Omega
 \end{array}\right.\qquad\label{2.23}
 \ees
Making use of \eqref{2.14}, \eqref{2.15} and \eqref{2.18}, we have
  \bess
d_1\ge\rho\ \ {\rm and}\ \ \max\kk\{ d_1,\  \frac{\sigma\chi}{R}|\nabla v|_{\alpha,\,\ol Q_T},\ \kk|-\frac{\sigma\chi}{R}\Delta v-a_1\rr|_{\alpha,\,\ol Q_T}\rr\}\le \sigma+\rho\sigma^2= h_4(\rho,\sigma).
  \eess
This enables us to apply Lemma \ref{l1.1}\,(ii) to \eqref{2.23} to derive
 \bes
|\tilde u|_{2+\alpha,\,\ol Q_T}\le \mathcal{L}(\alpha,h_4(\rho,\sigma),\rho)\left(\kk|\kk(\frac{\sigma c_1}{R}v-\frac{\sigma b_1}{R}\tilde u\rr)\tilde u\rr|_{\alpha,\,\ol Q_T}+\|\tilde u_0\|_{2+\alpha,\,\bar\Omega}+|\tilde u|_{0,\ol Q_T}\right).\label{3.10a}
 \ees
Similar to the derivation of \eqref{3.8}, by using \eqref{3.7a}, \eqref{2.18}, \eqref{2.15}, \eqref{2.19} and \eqref{2.22}, one can obtain
 \[\kk|\kk(\frac{\sigma c_1}{R}v-\frac{\sigma b_1}{R}\tilde u\rr)\tilde u\rr|_{\alpha,\,\ol Q_T}\le\sigma^3R(\rho+\sigma)(1+2P).\]
This combined with \eqref{3.10a}, \eqref{2.18a}, \eqref{2.19} and \eqref{2.2d} implies
\bess
|\tilde u|_{2+\alpha,\,\ol Q_T}&\le&\mathcal{L}(\alpha,h_4(\rho,\sigma),\rho)\kk(\kk|\kk(\frac{\sigma c_1}{R}v-\frac{\sigma b_1}{R}\tilde u\rr)\tilde u\rr|_{\alpha,\,\ol Q_T}+\|\tilde u_0\|_{2+\alpha,\,\bar\Omega}+|\tilde u|_{0,\ol Q_T}\rr)\\[1mm]
&\le&\mathcal{L}(\alpha,h_4(\rho,\sigma),\rho)[\sigma^3(\rho+\sigma)(1+2P)+2\sigma]R\\[1mm]
&=&[\sigma^3(\rho+\sigma)(1+2P)+2\sigma]\mathcal{L}_4R\\[1mm]
&\le&\rho.
\eess
This gives the last estimation of \eqref{3.10} and hence completes the proof.
\end{proof}

In what follows, we shall consider
\bes
 \left\{\begin{array}{lll}
 u_t=d_1\Delta u-\dd\frac{\sigma\chi}{R}\nabla\cdot(u \nabla v)+\kk(-a_1+\frac{\sigma c_1}{R} v-\frac{\sigma b_1}{R} u\rr) u,&x\in\Omega,\ \ t\in(0,T],\\[3mm]
  v_t=d_2\Delta v+\dd\frac{\sigma\xi}{R}\nabla\cdot (v\nabla u)+\kk(a_2-\frac{\sigma }{R}u-\frac{\sigma b_2 }{R}v\rr) v,&x\in\Omega,\ \ t\in(0,T],\\[3mm]
 \dd\frac{\partial u}{\partial\nu}=\dd\frac{\partial v}{\partial\nu}=0,\ \ &x\in\partial\Omega, \ \ t\in[0,T],\\[3mm]
 u(x,0)=\dd\frac{R}{\sigma}u_0(x),\ v(x,0)=\frac{R}{\sigma}v_0(x), &x\in\bar\Omega.
 \end{array}\right.\qquad\label{3.14}
 \ees
\begin{lem}\label{l3.3}\, Let $T\in[1,\infty),\alpha\in(0,1)$, and $\rho,\sigma,P,R$ be given by \eqref{2.1} and \eqref{2.2d}. Assume that
\bess
\chi\in(0,{\sigma R}/{3}]\ \ \ {\rm and}\ \ \xi\in(0,{R}/{3}].
\eess
Then there exists $(\tilde u,\tilde v)\in( C^{2+\alpha,1+\alpha/2}(\overline Q_T))^2$ which solves \eqref{3.14}. And $(\tilde u,\tilde v)$ satisfies
  \bess
0\le \tilde u\le \sigma R, \ \ |\tilde u|_{\alpha,\,\ol Q_T}\le \sigma PR, \ \ |\tilde u|_{2+\alpha,\,\ol Q_T}\le \rho,
\eess
and
  \bess
0\le \tilde v\le R, \ \  |\tilde v|_{\alpha,\,\ol Q_T}\le PR, \ \ |\tilde v|_{2+\alpha,\,\ol Q_T}\le \rho.
\eess
\end{lem}

\begin{proof}\, We first define
 \bess
\Sigma=\left\{u,v\in C^{2+\alpha,1+\alpha/2}(\overline Q_T):\,\begin{array}{ll} 0\le u\le \sigma R, \ 0\le v\le R, \ |u|_{\alpha,\,\ol Q_T}\le \sigma PR,\\[1mm]
 |v|_{\alpha,\,\ol Q_T}\le PR,\ \max\{|u|_{2+\alpha,\,\ol Q_T},\,|v|_{2+\alpha,\,\ol Q_T}\}\le \rho\end{array}\right\}
 \eess
Define
  \bess
  W=[C^{2+\alpha/2,1+\alpha/4}(\overline Q_T)]^2.
   \eess
It is easy to see that $W$ is a Banach space endowed with norm
\[\|(u,v)\|_W=|u|_{2+\alpha/2,\,\ol Q_T}+|v|_{2+\alpha/2,\,\ol Q_T}\ \ \ {\rm for}\ (u,v)\in W.\]
Moreover, $\Sigma$ is a compact convex subset of $W$.

For the given $(u,v)\in\Sigma$, let $\tilde v=M(u)$ be the unique solution of \eqref{3.1a} obtained by Lemma \ref{l3.1}, and $\tilde u=N(v)$ be the solution of \eqref{3.9a} given by Lemma \ref{l3.2}. Set $\mathcal{F}(u,v)=(\tilde u, \tilde v)$. It follows from Lemma \ref{l3.1} and Lemma \ref{l3.2} that $\tilde u,\tilde v\in C^{2+\alpha,1+\alpha/2}(\overline Q_T)$ and
\[0\le \tilde u\le \sigma R,\ \ \ |\tilde u|_{\alpha,\,\ol Q_T}\le \sigma PR,\ \ \ |\tilde u|_{2+\alpha,\,\ol Q_T}\le \rho,\]
and
\bes
0\le \tilde v\le R,\ \ \ |\tilde v|_{\alpha,\,\ol Q_T}\le PR,\ \ \ |\tilde v|_{2+\alpha,\,\ol Q_T}\le \rho.\label{3.11}
\ees
 Hence, $\mathcal{F}$ maps from $\Sigma$ into itself.

In order to apply the Schauder fixed point theorem, we will prove that $\mathcal{F}$ is continuous in the norm $\|\cdot\|_W$ of the Banach space $W$. For $(u_i,v_i)\in\Sigma$, $i=1,2$, let
\[\tilde v_i=M(u_i), \ \ \tilde u_i=N(v_i), \ \ u=u_1-u_2, \ \ v=v_1-v_2, \ \ \tilde u=\tilde u_1-\tilde u_2, \ \ \tilde v=\tilde v_1-\tilde v_2.\]
Clearly, $\tilde v=\tilde v_1-\tilde v_2$ satisfies
\bess
\left\{\begin{array}{lll}
 \tilde v_t-d_2\Delta \tilde v-\dd\frac{\sigma\xi}{R}\nabla u_1\cdot\nabla\tilde v-\kk(\frac{\sigma\xi}{R}\tr u_1+a_2-\frac{\sigma}{R}u_1-\frac{\sigma b_2}{R}(\tilde v_1+\tilde v_2)\rr)\tilde v\\[3mm]
\qquad =\dd\frac{\sigma\xi}{R}\nn \tilde v_2\cdot\nn u+\frac{\sigma\xi}{R}\tilde v_2\tr u-\frac{\sigma}{R} \tilde v_2u,\hspace{4mm}x\in\Omega,\ t\in(0,T],\\[3mm]
\dd\frac{\partial\tilde v}{\partial\nu}=0,\hspace{57mm}x\in\partial\Omega,\ t\in[0,T],\\[3mm]
 \tilde v(x,0)=0,\hspace{52mm}x\in\bar\Omega.
 \end{array}\right.
 \eess
It is clear that $|\varphi|_{\alpha/2,\,\ol Q_T}\le 3|\varphi|_{\alpha,\,\ol Q_T}$ for any $\varphi\in C^{\alpha,\alpha/2}(\overline Q_T)$. Since $\tilde v_i$ satisfy \eqref{3.11} for $i=1,2$ and $(u_1,v_1)\in \Sigma$, we have
\bess
|u_1|_{2+\alpha,\,\ol Q_T}\le \rho,\ \ |\tilde v_i|_{\alpha,\,\ol Q_T}\le PR\ \ {\rm for}\ i=1,2.
\eess
And hence there is $C_1>0$ such that
  \[\frac{\sigma\xi}{R}|\nabla u_1|_{\alpha/2,\,\ol Q_T}, \ \kk|\frac{\sigma\xi}{R}\tr u_1+a_2-\frac{\sigma}{R}u_1-\frac{\sigma b_2}{R}(\tilde v_1+\tilde v_2)\rr|_{\alpha/2,\,\ol Q_T}\le C_1.\]
Similarly, one can find $C_2>0$ such that
\[\kk|\frac{\sigma\xi}{R}\nn \tilde v_2\cdot\nn u+\frac{\sigma\xi}{R}\tilde v_2\tr u-\frac{\sigma}{R} \tilde v_2u\rr|_{\alpha/2,\,\ol Q_T}\le C_2|u|_{2+\alpha/2,\,\ol Q_T}.\]
In view of the parabolic Schauder theory (cf. \cite[Theorem IV.5.3]{L-S-Y1968}), there is $C_3>0$ which depends on $\alpha,T,\Omega,C_1$ such that
\bes
|\tilde v|_{2+\alpha/2,\,\ol Q_T}&\le& C_3\kk|\frac{\sigma\xi}{R}\nn \tilde v_2\cdot\nn u+\frac{\sigma\xi}{R}\tilde v_2\tr u-\frac{\sigma}{R} \tilde v_2u\rr|_{\alpha/2,\,\ol Q_T}\nonumber\\[1mm]
&\le&C_2C_3|u|_{2+\alpha/2,\,\ol Q_T}.\label{3.13}
\ees

We next estimate $\tilde u$. It is easy to see that $\tilde u=\tilde u_1-\tilde u_2$ satisfies
\bess
\left\{\begin{array}{lll}
 \tilde u_t-d_1\Delta\tilde u+\dd\frac{\sigma\chi}{R}\nabla v_1\cdot\nabla\tilde u+\kk(\frac{\sigma \chi}{R}\tr v_1+a_1-\frac{\sigma c_1}{R}v_1+\frac{\sigma b_1}{R}(\tilde u_1+\tilde u_2)\rr)\tilde u\\[3mm]
\qquad=-\dd\frac{\sigma\chi}{R}\nn\tilde u_2\cdot\nn v-\frac{\sigma\chi}{R}\tilde u_2\tr v+\frac{\sigma c_1}{R}\tilde u_2v,\hspace{4mm}x\in\Omega,\ \ 0<t\le T,\\[3mm]
 \dd\frac{\partial\tilde u}{\partial\nu}=0,\hspace{64mm}x\in\partial\Omega, \ \ 0\le t\le T,\\[3mm]
 \tilde u(x,0)=0, \hspace{58mm}x\in\bar\Omega.
 \end{array}\right.
 \eess
Similar to the above, there exist $C_4,C_5>0$ such that
\[\frac{\sigma\chi}{R}|\nabla v_1|_{\alpha/2,\,\ol Q_T},\ \ \kk|\frac{\sigma \chi}{R}\tr v_1+a_1-\frac{\sigma c_1}{R}v_1+\frac{\sigma b_1}{R}(\tilde u_1+\tilde u_2)\rr|_{\alpha/2,\,\ol Q_T}\le C_4,\]
and
\[\kk|-\frac{\sigma\chi}{R}\nn\tilde u_2\cdot\nn v-\frac{\sigma\chi}{R}\tilde u_2\tr v+\frac{\sigma c_1}{R}\tilde u_2v\rr|_{\alpha/2,\,\ol Q_T}\le C_5|v|_{2+\alpha/2,\,\ol Q_T}.\]
Again by the parabolic Schauder theory, there is $C_6>0$ such that
\bess
|\tilde u|_{2+\alpha/2,\,\ol Q_T}&\le& C_6|v|_{2+\alpha/2,\,\ol Q_T}.
\eess
This combined with \eqref{3.13} yields that
 \bess
\|\mathcal{F}(u_1,v_1)-\mathcal{F}(u_2,v_2)\|_W
&=&\|(\tilde u,\tilde v)\|_W\\[1mm]
&=&|\tilde u|_{2+\alpha/2,\,\ol Q_T}+|\tilde v|_{2+\alpha/2,\,\ol Q_T}\\[1mm]
&\le&\max\{C_2C_3,C_6\}\left(|u|_{2+\alpha/2,\,\ol Q_T}+|v|_{2+\alpha/2,\,\ol Q_T}\right)\\[1.5mm]
&\le&\max\{C_2C_3,C_6\}\|(u_1,v_1)-(u_2,v_2)\|_W.
 \eess
This shows that $\mathcal{F}$ is continuous in the norm $\|\cdot\|_W$ of the Banach space $W$.

Making use of the Schauder fixed point theorem (cf. \cite[Theorem 11.1]{GilbargT}), there exists $( \hat u, \hat v)\in \Sigma$ such that $\mathcal{F}(\hat u,\hat v)=(\hat u,\hat v)$. Hence,  problem \eqref{3.14} admits a solution $(\hat u,\hat v)$. And the desired estimates follow from the definition of $\Sigma$.
\end{proof}

In view of Lemma \ref{l3.3}, we can prove the existence of solutions of \eqref{1.1}.

\begin{lem}\label{l3.4}\, Let $\alpha\in(0,1)$, and $\rho,\sigma,P,R$ be given by \eqref{2.1} and \eqref{2.2d}. Assume that
\bess
\chi\in(0,{\sigma R}/{3}]\ \ \ {\rm and}\ \ \xi\in(0,{R}/{3}].
\eess
Then there exists a solution $(u,v)\in[C^{2,1}(\overline Q_T)]^2$ solving \eqref{1.1} on $[0,T]$ for any $T\in[1,\infty)$, and $(u,v)$ satisfies
\bes
0\le u\le \sigma^2,\ \ \ |u|_{\alpha,\,\ol Q_T}\le P\sigma^2,\ \ \ |u|_{2+\alpha,\,\ol Q_T}\le {\rho\sigma}/{R}.\label{3.17}
\ees
and
 \bes
0\le v\le \sigma,\ \ \ |v|_{\alpha,\,\ol Q_T}\le P\sigma,\ \ \ |v|_{2+\alpha,\,\ol Q_T}\le {\rho\sigma}/{R},\label{3.16}
 \ees
\end{lem}

\begin{proof}\, Let $T\in[1,\infty)$ and $(\hat u,\hat v)$ be the solution of \eqref{3.14} obtained in Lemma \ref{l3.3}, i.e.,
\bess
 \left\{\begin{array}{lll}
 \hat u_t=d_1\Delta  \hat u-\dd\frac{\sigma\chi}{R}\nabla\cdot(\hat u \nabla \hat v)+\kk(-a_1+\frac{\sigma c_1}{R}\hat v-\frac{\sigma b_1}{R} \hat u\rr) \hat u,\;\;&x\in\Omega,\ \ t\in(0,T],\\[3mm]
  \hat v_t=d_2\Delta\hat v+\dd\frac{\sigma\xi}{R}\nabla\cdot (\hat v\nabla \hat u)+\kk(a_2-\frac{\sigma }{R}\hat u-\frac{\sigma b_2 }{R}\hat v\rr)\hat v,\;\;&x\in\Omega,\ \ t\in(0,T],\\[3mm]
 \dd\frac{\partial \hat u}{\partial\nu}=\dd\frac{\partial \hat v}{\partial\nu}=0,\ \ &x\in\partial\Omega, \ \ t\in[0,T],\\[3mm]
 \hat u(x,0)=\dd\frac{R}{\sigma}u_0(x),\ \hat v(x,0)=\frac{R}{\sigma}v_0(x), &x\in\bar\Omega.
 \end{array}\right.
 \eess
By letting $u=\frac{\sigma}{R}\hat u$ and $v=\frac{\sigma}{R}\hat v$, then it is easy to verify that $(u,v)$ solves \eqref{1.1} on $[0,T]$ and fulfills \eqref{3.17} and \eqref{3.16}.
\end{proof}

\subsection{Uniqueness of solution}
The coming lemma asserts that the solution obtained in Lemma \ref{l3.4} is the unique solution of  \eqref{1.1}.

\begin{lem}\label{l3.5}
Let $\alpha\in(0,1)$, and $\rho,\sigma,P,R$ be given by \eqref{2.1} and \eqref{2.2d}. Assume that
\bes
\chi\in(0,\;{\sigma R}/{3}]\ \ \ {\rm and}\ \ \xi\in(0,\;{R}/{3}].\label{2.32}
\ees
Then there is a unique solution $(u,v)\in[C^{2,1}(\overline Q_T)]^2$ of the problem \eqref{1.1} on $[0,T]$ for $T\in[1,\yy)$.
\end{lem}
\begin{proof}
On the one hand, let $(u_1,v_1)\in[C^{2,1}(\overline Q_T)]^2$ be a solution of \eqref{1.1}. By using the maximum principle, it is easy to see that $u_1,v_1\ge0$. Clearly, there is $C_0>0$ such that
\bes
|u_1|_{2,\ol Q_T},\ |v_1|_{2,\ol Q_T}\le C_0.\label{3.15}
\ees
On the other hand, suppose that $(u_2,v_2)$ is the solution of \eqref{1.1} obtained in Lemma \ref{l3.4}. Then $(u_2,v_2)$ satisfies \eqref{3.17} and \eqref{3.16}, and hence
\bes
0\le u_2\le \sigma^2\ \ {\rm and}\ \ 0\le v_2\le \sigma.\label{3.17a}
\ees

Let $w=u_1-u_2$ and $z=v_1-v_2$. Then $w$ and $z$ satisfy
\bess
\left\{\begin{array}{lll}
 w_t-d_1\Delta w+\chi\nabla\cdot (w\nabla v_1)+[a_1-c_1v_1+b_1(u_1+u_2)]w\\[2mm]
\qquad=-\chi\nabla\cdot(u_2\nabla z)+c_1u_2z,\ \ &x\in\Omega,\ \ 0<t\le T,\\[2mm]
\dd\frac{\partial w}{\partial\nu}=0, &x\in\partial\Omega, \ \ 0\le t\le T,\\[2mm]
 w(x,0)=0, &x\in\bar\Omega,
 \end{array}\right.
 \eess
and
\bess
\left\{\begin{array}{lll}
 z_t-d_2\Delta z-\xi\nabla\cdot(z\nabla u_1)+(-a_2+u_1+b_2(v_1+v_2))z\\[2mm]
\qquad  =\xi\nabla\cdot(v_2\nabla w)-v_2w,\ \ &x\in\Omega,\ \ 0<t\le T,\\[2mm]
 \dd\frac{\partial z}{\partial\nu}=0,\ \ &x\in\partial\Omega, \ \ 0\le t\le T,\\[2mm]
z(x,0)=0, &x\in\bar\Omega.
 \end{array}\right.
 \eess
Making use of the testing procedure, by \eqref{2.1}, \eqref{2.32}, \eqref{3.17a} and \eqref{3.15}, we have that, for some $C_1,C_2>0$,
\bes
\dd\frac12\frac{d}{dt}\int_\Omega w^2{\rm d}x&=&-d_1\int_\Omega|\nabla w|^2{\rm d}x+\chi\int_\Omega w\nabla w\cdot\nabla v_1{\rm d}x+\chi\int_\Omega u_2\nabla z\cdot\nabla w{\rm d}x\nonumber\\[1mm]
&&-\int_\Omega w^2[a_1-c_1v_1+b_1(u_1+u_2)]{\rm d}x+c_1\int_\Omega u_2wz{\rm d}x\nonumber\\[1mm]
&\le&-\dd\frac{d_1}2\int_\Omega|\nabla w|^2{\rm d}x+\dd\frac{\chi^2}{2d_1}\int_\Omega w^2|\nabla v_1|^2{\rm d}x+\dd\frac{\chi}{2}\int_\Omega u_2|\nabla w|^2{\rm d}x\nonumber\\[1mm]
&&+\dd\frac{\chi}{2}\int_\Omega u_2|\nabla z|^2{\rm d}x+c_1\int_\Omega v_1w^2{\rm d}x+c_1\int_\Omega u_2wz{\rm d}x\nonumber\\[1mm]
&\le&\left(-\dd\frac{\rho}2+\dd\frac{\sigma^3 R}{6}\right)\int_\Omega|\nabla w|^2{\rm d}x+\dd\frac{\sigma^3 R}{6}\int_\Omega |\nabla z|^2{\rm d}x+C_1\int_\Omega(w^2+z^2){\rm d}x,\label{3.18}
\ees
and
\bes
\dd\frac12\frac{d}{dt}\int_\Omega z^2{\rm d}x&=&-d_2\int_\Omega|\nabla z|^2{\rm d}x-\xi\int_\Omega z\nabla z\cdot\nabla u_1{\rm d}x-\xi\int_\Omega v_2\nabla w\cdot\nabla z{\rm d}x\nonumber\\[1mm]
&&-\int_\Omega z^2(-a_2+u_1+b_2(v_1+v_2)){\rm d}x-\int_\Omega v_2wz{\rm d}x\nonumber\\[1mm]
&\le&-\dd\frac{d_2}{2}\int_\Omega|\nabla z|^2{\rm d}x+\dd\frac{\xi^2}{2d_2}\int_\Omega z^2|\nabla u_1|^2{\rm d}x+\dd\frac{\xi}{2}\int_\Omega v_2|\nabla w|^2{\rm d}x \nonumber\\[1mm]
&&+\dd\frac{\xi}{2}\int_\Omega v_2|\nabla z|^2{\rm d}x+a_2\int_\Omega z^2{\rm d}x-\int_\Omega v_2wz{\rm d}x\nonumber\\[1mm]
&\le&\left(-\dd\frac{\rho}{2}+\dd\frac{\sigma R}{6}\right)\int_\Omega|\nabla z|^2{\rm d}x+\dd\frac{\sigma R}{6}\int_\Omega |\nabla w|^2+C_2\int_\Omega(w^2+z^2){\rm d}x.\label{3.19}
\ees
It follows from \eqref{3.18} and \eqref{3.19} that
\bes
\dd\frac12\frac{d}{dt}\int_\Omega (w^2+z^2){\rm d}x&\le&\left(-\dd\frac{\rho}2+\dd\frac{\sigma R}{6}+\dd\frac{\sigma^3R}{6}\right)\kk(\int_\Omega|\nabla w|^2{\rm d}x+\int_\Omega|\nabla z|^2{\rm d}x\rr)\nonumber\\[1mm]
&&+(C_1+C_2)\int_\Omega(w^2+z^2){\rm d}x. \label{3.20}
\ees
From the definition of $R$, we have $R\le 3\rho/(\sigma+\sigma^3)$ and hence
\[-\dd\frac{\rho}2+\dd\frac{\sigma R}{6}+\dd\frac{\sigma^3R}{6}\le0.\]
This combined with \eqref{3.20} yields
\bess
\frac{d}{dt}\int_\Omega (w^2+z^2){\rm d}x\le2(C_1+C_2)\int_\Omega(w^2+z^2){\rm d}x.
\eess
Noting that $w(x,0)=z(x,0)=0$, Gronwall's lemma asserts that $w=z=0$. This completes the proof.
\end{proof}

\begin{proof}[\bf Proof of Theorem {\rm\ref{t1}}]
Combining the conclusions of Lemmas \ref{l3.4}, \ref{l3.5}, and using the arbitrariness of $T\ge 1$ we know that the solution $(u,v)$ of \eqref{1.1} exists uniquely and globally, and the estimate \eqref{1.2} holds.
\end{proof}

 \end{document}